\documentclass{amsart}
\usepackage{amsfonts,amscd,amsthm,amsgen,amsmath,amssymb}
\usepackage[all]{xy}
\usepackage{epsfig,color}
\usepackage[vcentermath]{youngtab}
\newtheorem{theorem}{Theorem}[section]
\newtheorem{lemma}[theorem]{Lemma}

\newtheorem{proposition}[theorem]{Proposition}

\theoremstyle{definition}
\newtheorem{definition}[theorem]{Definition}

\theoremstyle{remark}
\newtheorem{remark}[theorem]{Remark}

\numberwithin{equation}{section}

\begin{document}

\setlength\parskip{0.5em plus 0.1em minus 0.2em}

\title{Exotic diffeomorphisms of reducible $4$-manifolds with odd $b_+$}
\author{David Baraglia}

\address{School of Mathematical Sciences, Adelaide University, Adelaide SA 5005, Australia}

\email{david.baraglia@adelaide.edu.au}


\date{\today}

\begin{abstract}

A diffeomorphism of a $4$-manifold is said to be exotic if it is continuously isotopic to the identity but not smoothly isotopic to the identity. Ruberman constructed the first examples of exotic diffeomorphisms on simply-connected closed $4$-manifolds. His examples were reducible $4$-manifolds that necessarily have even $b_+$ in order that they can be detected by the families Seiberg--Witten or Donaldson invariants. Later Konno and Baraglia produced exotic diffeomorphisms on irreducible $4$-manifolds with odd $b_+$. In this paper, we will construct exotic diffeomorphisms on reducible $4$-manifolds with odd $b_+$. Exoticness is detected using a families Bauer--Furuta invariant. In proving our results we need to work with families moduli spaces which are not framed and so do not give rise to framed cobordism invariants. We overcome this difficulty by considering a Bauer--Furuta type invariant valued in {\em pin-cobordism}. In addition to constructing exotic diffeomorphisms, we also find new examples of simply-connected $4$-manifolds whose mapping class groups are not finitely generated.

\end{abstract}

\maketitle




\section{Introduction}

Let $X$ be a compact, oriented, smooth, simply-connected $4$-manifold. The {\em mapping class group of $X$}, $M(X)$ is defined as the group of smooth isotopy classes of orientation preserving diffeomorphisms of $X$. The {\em Torelli group of $X$}, $T(X)$ is the subgroup of $M(X)$ consisting of isotopy classes which are continuously isotopic to the identity. The Torelli group measures the difference between smooth and continuous isotopy. In particular, $T(X)$ is non-trivial if and only if $X$ admits exotic diffeomorphisms, where a diffeomorphism is said to be {\em exotic} if it is continuously isotopic to the identity, but not smoothly isotopic to the identity.

The first examples of exotic diffeomorphisms on compact simply-connected $4$-manifolds were constructed by Ruberman \cite{rub1,rub2}. His examples were of the form $2n \mathbb{CP}^2 \# k \overline{\mathbb{CP}^2}$ with $n \ge 2$, $k \ge 10n+1$. Exoticness is detected using families Seiberg--Witten or families Donaldson invariants. Baraglia--Konno extended Ruberman's technique, allowing for spin examples such as $n(S^2 \times S^2) \# nK3$, $n \ge 1$ \cite{bk1}. In all of these examples $b_+(X)$ is even, which is a necessary consequence of the fact that the diffeomorphisms are being detected using a $1$-parameter families Seiberg--Witten invariant. Later, Baraglia and Konno found a method of constructing exotic diffeomorphisms on certain irreducible, simply-connected $4$-manifolds with odd $b_+$ \cite{bk2}. This method however does not work for connected sums such as $a \mathbb{CP}^2 \# b \overline{\mathbb{CP}^2}$ and $a (S^2 \times S^2) \# b K3$ because the manifold in question needs to have a non-vanishing Seiberg--Witten invariant.

The aim of this paper is to construct exotic diffeomorphisms on reducible $4$-manifolds with odd $b_+$. Our methods also produce additional results of interest including exotic diffeomorphisms which remain exotic on connected sum with a large class of $4$-manifolds and new examples of simply-connected $4$-manifolds for which the mapping class group is not finitely generated. Our main results are as follows:

\begin{theorem}
Let $X$ be one of:
\begin{itemize}
\item[(1)]{$n (S^2 \times S^2) \# (n+1) K3$, $n \ge 1$, or}
\item[(2)]{$(4 n+3) \mathbb{CP}^2 \# k \overline{\mathbb{CP}^2}$, $n\ge 1$, $k \ge 20n+19$.}
\end{itemize}

Then there exists a surjective homomorphism
\[
\varphi : T(X) \to \mathbb{Z}_2^\infty.
\]
In particular, the Torelli group $T(X)$ is not finitely generated.
\end{theorem}

The next theorem shows that there are exotic diffeomorphisms that remain exotic on connected sum with a large class of $4$-manifolds. It is a simplified version of Theorem \ref{thm:exotic2}. In this theorem $T_0(X)$ denotes the relative Torelli group of $X$ (see Section \ref{sec:exoticdiff} for the definition).

\begin{theorem}
Let $X$ be one of:
\begin{itemize}
\item[(1)]{$n (S^2 \times S^2) \# n K3$, $n \ge 1$, or}
\item[(2)]{$4 n \mathbb{CP}^2 \# k \overline{\mathbb{CP}^2}$, $n\ge 1$, $k \ge 20n$.}
\end{itemize}

Then there exists a homomorphism
\[
\Phi : T(X) \to \mathbb{Z}_2^\infty
\]
whose image is not finitely generated and has the following property. Let $X_1,X_2$ be compact, oriented smooth simply-connected $4$-manifolds with $b_+(X_i) = 3 \; ({\rm mod} \; 4)$ and $SW(X_i , \mathfrak{s}_i) = 1 \; ({\rm mod} \; 2)$ for some spin$^c$-structure $\mathfrak{s}_i$. Then there exists homomorphisms
\[
\Psi_1 : T(X \# X_1 ) \to \mathbb{Z}_2^\infty, \quad \Psi_2 : T(X \# X_1 \# X_2) \to \mathbb{Z}_{24}^\infty
\]
such that the following diagrams commute
\[
\xymatrix{
T_0(X) \ar[r]^-{ \# id_{X_1} } \ar[d] & T(X \# X_1) \ar[d]^-{\Psi_1} &  T_0(X) \ar[rr]^-{ \# id_{(X_1 \# X_2)}} \ar[d] & & T(X \# X_1 \# X_2) \ar[d]^-{\Psi_2} \\
T(X) \ar[r]^-{\Phi} & \mathbb{Z}_2^\infty & T(X) \ar[r]^-{\Phi} & \mathbb{Z}_2^\infty \ar[r]^-{12} & \mathbb{Z}_{24}^\infty
}
\]

In particular, $X$ admits infinitely many isotopy classes of exotic diffeomorphisms which remain exotic on connected sums with $X_1$ and $X_1 \# X_2$.
\end{theorem}

Our last main result is that the mapping class groups of certain simply-connected $4$-manifolds are not finitely generated. In \cite{bar2}, we proved that $M(X)$ was not finitely generated for $X = 2n \mathbb{CP}^2 \# 10n \overline{\mathbb{CP}}^2 \cong E(n) \# (S^2 \times S^2)$ for any odd $n \ge 3$. The same result was shown by Konno for $E(n) \# (S^2 \times S^2)$ and $n \ge 2$ \cite{kon}. In \cite{bt}, we extended this result to the case $n=1$. We will prove the following:

\begin{theorem}

Let $X$ be one of 
\begin{itemize}
\item[(1)]{$4 n \mathbb{CP}^2 \# 20n \overline{\mathbb{CP}^2}$, $n\ge 1$.}
\item[(2)]{$(4 n-1) \mathbb{CP}^2 \# (20n-1) \overline{\mathbb{CP}^2}$, $n\ge 2$.}
\end{itemize}

Then $M(X)$ is not finitely generated. In fact, the abelianisation of $M(X)$ is not finitely generated.

\end{theorem}

We briefly outline the key ideas used in this paper. As mentioned earlier, familes Seiberg--Witten invariants can be used to detect exotic diffeomorphisms, but due to dimensional constraints this requires $b_+$ to be odd (assuming $b_1 = 0$). To get around this problem we instead consider the (non-equivariant) families Bauer--Furuta invariant. This is an invariant for families of $4$-manifolds which is valued in stable cohomotopy. Under the Pontryagin--Thom construction, the families Bauer--Furuta invariant is the framed cobordism class of the families moduli space of a finite-dimensional approximation of the Seiberg--Witten equations. However, such an invariant is only defined under certain restrictions which ensures that the families moduli space has a stable framing. Our method to obtain exotic diffeomorphisms necessitates considering moduli spaces which do not have a natural choice of framing. The main innovation used in this paper is to consider a Bauer--Furuta type invariant valued in {\em pin-cobordism} (more precisely, $Pin^-$-cobordism in the terminology of \cite{kt}). Using this invariant, we can carry over existing constructions of exotic diffeomorphisms to a new class of examples.

\subsection{Structure of the paper}

The structure of the paper is as follows. In Section \ref{sec:moduli} we define a (non-equivariant) families Bauer--Furuta invariant of families of $4$-manifolds valued in framed cobordism. In Section \ref{sec:fbf} we specialise to the mapping torus of a diffeomorphism to obtain a framed cobordism-valued invariant of diffeomorphisms. We prove connected sum and blowup formulas for this invariant. In Section \ref{sec:pin} we introduce a families Bauer--Furuta invariant which is instead valued in pin-cobordism. In Section \ref{sec:exoticdiff}, we use our families Bauer--Furuta invariants to detect exotic diffeomorphisms. Finally in Section \ref{sec:mcg}, we use our invariants to prove that the mapping class group of certain simply-connected $4$-manifolds is not finitely generated.

\noindent{\bf Acknowledgments.} The author was financially supported by an Australian Research Council Future Fellowship, FT230100092.

\section{$1$-parameter Bauer--Furuta invariants}

\subsection{Families Bauer--Furuta maps and framed cobordism classes}\label{sec:moduli}

We begin with a short review of the Bauer--Furuta map and families Bauer--Furuta map following \cite{bf,bk}.

First consider the case of a single $4$-manifold. Let $X$ be a compact, oriented, smooth $4$-manifold with $b_1(X) = 0$. For a spin$^c$-structure $\mathfrak{s}$ on $X$, let
\[
d(\mathfrak{s}) = \frac{ c(\mathfrak{s})^2 - \sigma(X) }{4} - b_+(X) - 1
\]
denote the expected dimension of the Seiberg--Witten moduli space for $(X,\mathfrak{s})$. Let $g$ be a Riemannian metric on $X$, let $S^{\pm}$ denote the spinor bundles corresponding to $\mathfrak{s}$ and let $A$ denote a spin$^c$-connection. The {\em Seiberg--Witten monopole map} for $(X,\mathfrak{s})$ is a non-linear Fredholm map
\[
SW : \mathbb{W} \to \mathbb{W}' 
\]
where
\begin{align}
\mathbb{W} &= L^2_k(X , S^+) \oplus L^2_k(X , T^*X),  \label{equ:w} \\
\mathbb{W}' &= L^2_{k-1}(X , S^-) \oplus L^2_{k-1}(X , \wedge^2_+ T^*X) \oplus L^2_{k-1}(X , \mathbb{R})_0. \label{equ:w'}
\end{align}
Here $k > 2$ is an integer and $L^2_{k-1}(X , \mathbb{R})_0$ denotes the subspace of $L^2_{k-1}(X , \mathbb{R})_0$ which is $L^2$-orthogonal to constant functions. The map is given by
\[
SW( \psi , a )  = ( D_{A+ia} \psi , d^+ a + i\sigma(\psi) , d^* a)
\]
where $D_{A+ia}$ is the Dirac operator corresponding to $A+ia$ and $\sigma(\psi) = \rho^{-1}( (\psi \otimes \psi^*)_0)$ is the self-dual $2$-form corresponding to the trace-free part of $\psi \otimes \psi^*$ under Clifford multiplication. The monopole map can be written as the sum of a linear Fredholm map and a non-linear term, namely $SW = l + c$, where
\[
l(\psi , a) = (D_A \psi , d^+ a , d^*a), \quad c(\psi , a) = ( ia \psi , i\sigma(\psi) , 0).
\]

The map $SW$ is $S^1$-equivariant for certain circle actions on $\mathbb{W}, \mathbb{W}'$. However in this paper we will not make use of this circle action.

Now let $W' \subset \mathbb{W}$ be a finite dimensional subspace of $\mathbb{W}$ such that $\mathbb{W} = W' + l(\mathbb{W})$ and set $W = l^{-1}(W')$. In \cite{bf} a deformation retract $\rho_{W'} : (\mathbb{W}')^+ \setminus S( W'^{\perp} ) \to (W')^+$ is constructed, where $S(W'^\perp)$ denotes the unit sphere in $W'^{\perp}$ and the $+$ superscripts denote one-point compactifications. It is further shown that if $W'$ is chosen sufficiently large, then $SW|_{W}$ misses $S(W'^{\perp})$ and so the composition
\[
f = \rho_{W'} \circ SW : W^+ \to W'^+
\]
is defined. Bauer and Furuta show that if $W'$ is sufficiently large then the stable homotopy class of $f$ does not depend on the choice of $W'$, nor on the choice of metric $g$ and reference connection $A$ and thus defines an invariant of $(X,\mathfrak{s})$ alone. This is the (non-equivariant) {\em Bauer--Furuta invariant} of $(X,\mathfrak{s})$:
\[
BF_{X,\mathfrak{s}} = [f] \in \pi^{s}_n,
\]
where $\pi^{s}_*$ denotes the stable homotopy groups of spheres and $n = dim(W) - dim(W') = ind(l) = d(\mathfrak{s})+1$. Strictly speaking, to identify the stable homotopy class of $f$ with an element of $\pi^s_n$ requires a choice of trivialisations $W \cong \mathbb{R}^a$, $W' \cong \mathbb{R}^{a'}$. Up to homotopy, there are two choices of trivialisation which correspond to the two possible orientations on $W,W'$. Moreover only the overall orientation on $W \oplus W'$ matters. We have a stable isomorphism $W - W' = ker(D_A) - coker(D_A) - H^+(X)$. Then since $ker(D_A)$, $coker(D_A)$ are complex vector spaces, they come with natural orientations. So an overall orientation on $W \oplus W'$ is equivalent to a choice of orientation of $H^+(X)$.

Under the Pontryagin--Thom isomorphism $\pi^{s}_n \cong \Omega^{fr}_n$, where $\Omega^{fr}_*$ denotes the framed cobordism ring. So we often regard the Bauer--Furuta invariant as a framed cobordism class
\[
BF_{X,\mathfrak{s}} \in \Omega^{fr}_{d(\mathfrak{s})+1}.
\]
Thought of this way, $BF_{X , \mathfrak{s}}$ is the cobordism class of the ``moduli space" $\mathcal{M} = f^{-1}(\mu)$, where $\mu$ is a regular value of $f$.

Now we consider the Bauer--Furuta invariant in the setting of families. Let $B$ be a compact smooth manifold and $\pi : E \to B$ a smooth fibre bundle with a fibrewise orientation such that the fibres of $E$ are orientation preservingly diffeomorphic to $X$. By a families spin$^c$-structure on $E$, we mean a spin$^c$-structure $\mathfrak{s}$ on the vertical tangent bundle $V = Ker(\pi_*)$. Following \cite{bk}, we can construct the families Seiberg--Witten monopole map
\[
SW : \mathbb{W} \to \mathbb{W}'
\]
where $\mathbb{W}, \mathbb{W}'$ are now Hilbert bundles over $B$ whose fibres are given by (\ref{equ:w}), (\ref{equ:w'}). Once again we can decompose $SW$ into the sum of a linear Fredholm operator and a non-linear term, namely $SW = l + c$, where $l = D \oplus (d^+ , d^*)$ is the direct sum of the Dirac operator $D$ (with respect to some family of reference spin$^c$-connections) and the Atiyah--Hitchin--Singer operator $(d^+ , d^*)$. Let $W' \subset \mathbb{W}'$ be a finite-dimensional subbundle of $\mathbb{W}'$ such that for each $b \in B$, $\mathbb{W}'_b = W'_b + l_b(\mathbb{W}_b)$ and set $W = l^{-1}(W')$. For a vector bundle $V \to B$, lset $S_B^V$ denote the unit sphere bundle of $V \oplus \mathbb{R}$. Equivalently, $S^V_B$ is the sphere bundle over $B$ obtained by taking $1$-point compactifications of each fibre of $V$. Similar to the case of a single $4$-manifold it can be shown that if $W'$ is chosen sufficiently large, then $SW|_{W}$ misses $S(W'^\perp)$ and so the composition 
\[
f = \rho_{W'} \circ SW : S^W_B \to S_B^{W'}
\]
is a well-defined map of sphere bundles over $B$. Furthermore, if $W'$ is sufficiently large, then the stable homotopy class of $f$ is independent of the choice of $W'$. Here by homotopy, we mean homotopy through maps of sphere bundles and by stabilisation, we mean the operation that replaces $f$ by the fibrewise stabilisation $f \wedge_B id_{S^U_B} : S^{W \oplus U}_B \to S_B^{W' \oplus U}$ for any vector bundle $U$. For convenience we will sometimes refer to $f : S^W_B \to S^{W'}_B$ as a {\em Bauer--Furuta map} for the family $(E , \mathfrak{s})$. It is possible to reduce $f$ to a stable homotopy class in the ordinary sense as follows. Choose a vector bundle $U$ such that $W' \oplus U$ is trivial. Then upon stabilising by $U$, we can assume that $W'$ is trivial and so $f$ can be regarded as a stable homotopy class of map $S^W_B \to S^k$, where $k = dim(W')$.

We have an equality of $K$-theory classes in $K^0(B)$:
\begin{equation}\label{equ:ww'}
W - W' = ind(l) = ind(D) - H^+(E/B),
\end{equation}
where $ind(D)$ denotes the families index of the Dirac operator corresponding to $\mathfrak{s}$ and $H^+(E/B)$ is the vector bundle whose fibre over $b \in B$ is equal to $H^+(X_b)$, the space of harmonic self-dual $2$-forms on $X_b = \pi^{-1}(b)$.

Choose a generic section $\mu : B \to S^{W'}_B$ that is transverse to $f$ and set $\mathcal{M} = f^{-1}(\mu)$. We will sometimes refer to $\mathcal{M}$ as a {\em moduli space} for the family $(E , \mathfrak{s})$, since it is a kind of approximation of the families Seiberg--Witten moduli space. If we wish to show the dependence of $\mathcal{M}$ on the choice of $\mu$ and the choice of finite-dimensional approximation $f$, we may write $\mathcal{M} = \mathcal{M}(f,\mu)$.

The normal bundle of $\mu$ can be identified with $W'$ and so we obtain an exact sequence
\begin{equation}\label{equ:tm}
0 \to T\mathcal{M} \to TS^W_B |_{\mathcal{M}} \to W' \to 0.
\end{equation}
Up to stable isomorphism we have $TS^W_B \cong TB \oplus W$ and hence $T\mathcal{M}$ is stably isomorphic to
\[
TB + W - W' = TB + ind(D) - H^+(E/B).
\]
Suppose now that $ind(D)$, $H^+(E/B)$ and $TB$ are all trivial. Then $T\mathcal{M}$ is stably trivial and $\mathcal{M}$ defines a framed cobordism class $[\mathcal{M}] \in \Omega^{fr}_{k}$, where $k = d(\mathfrak{s}) + 1 + dim(B)$. The cobordism class $[\mathcal{M}] = [\mathcal{M}(f,\mu)]$ may in general depend on the homotopy class of $\mu$ and also on the choice of trivialisations of $ind(D), H^+(E/B)$, so it does not quite define an invariant of the family $E$. In what follows we will eventually restrict to the case of $1$-parameter families where $B = S^1$, but let us first make some general observations about the depence on $[\mathcal{M}]$ on the choices made.

We need to be more precise about how trivialisations of $ind(D)$, $H^+(E/B)$ and $TB$ determine a stable framing of $\mathcal{M}(f,\mu)$. We will assume that the image of $\mu : B \to S^{W'}_B$ is disjoint from the section at infinity. Such a section can always be chosen since $S^{W'}_B$ minus the section at infinity is the vector bundle $W' \to B$. This also shows that any two sections of $S^{W'}_B$ which are disjoint from infinity are homotopic. Next, $ind(D)$ can be represented as $V-V'$ for some fixed complex vector bundles $V,V'$ on $B$ (not depending on $f$ or $\mu$). The equality (\ref{equ:ww'}) in $K$-theory can be expressed as an isomorphism of vector bundles on $B$ of the form
\[
\varphi_f : W \oplus V' \oplus H^+(E/B) \oplus \mathbb{R}^n \to V \oplus W' \oplus \mathbb{R}^n.
\]
Moreover $\varphi_f$ is compatible with stabilisation in the sense that if $W,W'$ are replaced with $W \oplus U, W' \oplus U$ for some vector bundle $U \to B$ and $f$ is replaced by $f_U = f \wedge_B id_{S^U_B} : S_B^{W \oplus U} \to S_B^{W' \oplus U}$, then $\varphi_{f_U} = \varphi_{f} \oplus id_U$. 

Since $S^W_B$ is the unit sphere bundle of $W \oplus \mathbb{R}$, we have an isomorphism $TS^W_B \oplus \mathbb{R} \cong TB \oplus W \oplus \mathbb{R}$. The exact sequence (\ref{equ:tm}) then gives an isomorphism
\[
\rho_{f} : TB \oplus W \oplus \mathbb{R} \to TS^W_B \oplus \mathbb{R} \cong T \mathcal{M}(f,\mu) \oplus W' \oplus \mathbb{R}.
\]
The map $\rho_{f}$ is compatible with stabilisation in the following sense. If $f$ is replaced by $f_U = f \wedge_B id_{S^U_B}$, then $\mu
$ can be regarded as a section of $S_B^{W' \oplus U}$ and we clearly have that $\mathcal{M}(f_U , \mu)$ is equal to the image of $\mathcal{M}(f , \mu)$ under the inclusion map $S^W_B \to S^{W \oplus U}_B$. So we will identify $\mathcal{M}(f_U , \mu)$ with $\mathcal{M}(f_U , \mu)$ without further mention. We then have that $\rho_{f_U} = \rho_f \oplus id_U$.

Composing $\rho_{f}^{-1}$ and $\varphi_{f}$, we get an isomorphism
\begin{align*}
T\mathcal{M}(f,\mu) \oplus W' \oplus \mathbb{R} \oplus V' \oplus H^+(E/B) \oplus \mathbb{R}^n &\to TB \oplus W \oplus \mathbb{R} \oplus V' \oplus H^+(E/B) \oplus \mathbb{R}^n \\
& \to TB \oplus V \oplus W' \oplus \mathbb{R}^n \oplus \mathbb{R}.
\end{align*}

Given trivialisations of $V,V', H^+(E/B)$ and $TB$ we then get a stable trivialisation
\[
\tau_{f} : T\mathcal{M}(f,\mu) \oplus \mathbb{R}^{u} \oplus W' \to \mathbb{R}^{v} \oplus W'
\]
where $u = dim(V') + dim(H^+(E/B)) + n+1$, $v = dim(V) + n + 1$. The isomorphism $\tau_{f}$ is compatible with stabilisation since $\varphi_{f}$ and $\rho_{f}$ are, so we get a stable trivialisation of $T\mathcal{M}(f,\mu)$ that is compatible with stabilisation $f \mapsto f_U$. It follows that the framed cobordism class $[\mathcal{M}(f,\mu)] \in \Omega^{fr}_*$ does not depend on the choice of finite-dimensional approximation $f$.

We have already seen that any two sections $\mu,\mu'$ of $S^{W'}_B$ which are disjoint from infinity are homotopic. This implies that $[\mathcal{M}(f,\mu)] = [\mathcal{M}(f,\mu')]$. In more detail, suppose $\mu,\mu'$ are two sections of $S^{W'}_B$ which are transverse to $f$ and disjoint from infinity. By standard differential topology we can find a section $\mu_t$ of the pullback $S^{W'}_{[0,1] \times B}$ of $S^{W'}_B$ to $[0,1] \times B$ such that $\mu_0 = \mu$, $\mu_1 = \mu'$, $\mu_t$ is disjoint from infinity and $\mu_t$ is transverse to $\widehat{f}$, where $\widehat{f}$ is the map $S^W_{[0,1] \times B} \to S^{W'}_{[0,1] \times B}$ which equals $f$ for each $t \in [0,1]$. Then $\mathcal{M}(\widehat{f} , \mu_t )$ is a cobordism from $\mathcal{M}(f,\mu_0)$ to $\mathcal{M}(f,\mu_1)$. Using the procedure described above we get a stable framing on $\mathcal{M}(\widehat{f} , \mu_t)$ which agrees with the stable framings on $\mathcal{M}(f,\mu_0), \mathcal{M}(f,\mu_1)$ on the endpoints. Hence $[\mathcal{M}(f,\mu_0)] = [\mathcal{M}(f,\mu_1)]$ in $\Omega^{fr}_*$.

To summarise, we have constructed a framed cobordism class $[\mathcal{M}(f,\mu)] \in \Omega^{fr}_k$, where $k = d(\mathfrak{s}) + 1 + dim(B)$ which depends only on trivialisations of $V,V',H^+(E/B)$ and $TB$. In the case $B = S^1$, $V,V'$ are always trivial, because they are complex vector bundles bundle and of course $TS^1$ is trivial. So if $H^+(E/B)$ is trivial, then we obtain a framed cobordism class $[\mathcal{M}] \in \Omega^{fr}_{d(\mathfrak{s})+2}$. Since $[S^1 , GL(n , \mathbb{R}) ] = \pi_1(GL(n,\mathbb{R})) \cong \mathbb{Z}_2$ for all $n > 2$, there is up to homotopy two stable framings of $\mathcal{M}$ obtained in the manner described above. Let $\xi,\xi'$ denote the two framings and $[\mathcal{M}(\xi)], [\mathcal{M}(\xi')] \in \Omega^{fr}_{d(\mathfrak{s})+2}$ the two corresponding framed cobordism classes.

\begin{proposition}\label{prop:xixi'}
Suppose that $B = S^1$ and that $H^+(E/B)$ is trivial. Then
\[
[\mathcal{M}(\xi)] - [\mathcal{M}(\xi')] = \eta \cdot BF_{X,\mathfrak{s}} \in \Omega^{fr}_{d(\mathfrak{s})+2}
\]
where $\eta$ denotes the generator of $\Omega^{fr}_1 \cong \mathbb{Z}_2$.
\end{proposition}
\begin{remark}
Notice that since $\eta$ is $2$-torsion, the order in which we choose the two stable framings $\xi,\xi'$ is irrelevant.
\end{remark}
\begin{proof}
Take the family $E \to B = S^1$ and pull it back to the cylinder $S^1 \times [0,1]$. The base $S^1 \times [0,1]$ can be viewed as a cobordism from $S^1$ to $S^1$. Now let $\Sigma$ be obtained by removing a ball from $S^1 \times [0,1]$. We can view $\Sigma$ as a sphere with three discs removed. Let $E_\Sigma \to \Sigma$ be the family obtained by pulling $E$ back to $S^1 \times [0,1]$ and then restricting to $\Sigma$. We think of $\Sigma$ as a cobordism with two ingoing circle boundaries $S^1_1, S^1_2$ and one outgoing boundary component $S^1_3$, where $S^1,S^1_3$ are the boundary components of $S^1 \times [0,1]$ and $S^1_2$ is the third boundary component resulting from removing a disc. Then $E|_{S^1_1} \cong E|_{S^1_3} \cong E$, whereas $E|_{S^1_2}$ is the trivial family $X \times S^1$.

Let $f_\Sigma$ denote the families Bauer--Furuta map for $E_\Sigma$ and $\mathcal{M}_\Sigma = f_\Sigma^{-1}(\mu_\Sigma)$ the corresponding moduli space. Then $\mathcal{M}_\Sigma$ is a cobordism from $\mathcal{M}_{\Sigma}|_{S^1_1} \cup \mathcal{M}_{\Sigma}|_{S^1_2}$ to $\mathcal{M}_{\Sigma}|_{S^1_3}$. The map $f_\Sigma$ can be obtained by taking the families Bauer--Furuta map $f$ for $E$, pulling it back to $S^1 \times [0,1]$ and then restricting to $\Sigma$. If $\mathcal{M} = f^{-1}(\mu)$ denotes the moduli space for $f$, then choosing $\mu_\Sigma$ to be the restriction of the pullback of $\mu$, it follows that $\mathcal{M}_\Sigma$ is given by taking $\mathcal{M}$, pulling it back to $S^1 \times [0,1]$ and restricting to $\Sigma$. Therefore we can assume that $\mathcal{M}_\Sigma |_{S^1_1} = \mathcal{M}$, $\mathcal{M}_\Sigma|_{S^1_3} = \mathcal{M}$ and $\mathcal{M}|_{S^1_2} = \mathcal{M}_0 \times S^1_2$, where $\mathcal{M}_0$ denotes the moduli space for the unparametrised Bauer--Furuta map of $(X , \mathfrak{s})$.

Fix an orientation on $H^+(X)$. This determines a framing $\xi_0$ on the unparametrised moduli space $\mathcal{M}_0$ such that the class of $[\mathcal{M}_0 , \xi_0] \in \Omega^{fr}_{d(\mathfrak{s})+1}$ is equal to $BF_{X,\mathfrak{s}}$. Let $\hat{\xi}_0 = \xi_0 \times \xi_{triv}$ denote the framing of $\mathcal{M}_0 \times S^1_2$ which is given by the product of the framing $\xi_0$ on $\mathcal{M}_0$ and the trivial framing $\xi_{triv}$ of $S^1_2$ (the framing that extends over the disc bounding $S^1_2$).

Let $\xi,\xi'$ denote the two framings on $\mathcal{M}$ constructed as above. It is easily seen that there are corresponding framings $\xi_\Sigma, \xi'_\Sigma$ on $\mathcal{M}_\Sigma$ with the property that $\xi_\Sigma |_{S^1_1} = \xi$, $\xi_\Sigma|_{S^1_3} = \xi$, $\xi_\Sigma |_{S^1_2} = \hat{\xi} \times \xi_{triv}$. Similarly $\xi'_\Sigma |_{S^1_1} = \xi'$, $\xi'_\Sigma|_{S^1_3} = \xi'$, $\xi'_\Sigma |_{S^1_2} = \hat{\xi} \times \xi_{triv}$. We have that $H_1(\Sigma ; \mathbb{Z}) \cong \mathbb{Z}^2$, generated by $[S^1_1], [S^1_2], [S^1_3]$ subject to the relation $[S^1_3] = [S^1_1] + [S^1_2]$. The set of framings on $\mathcal{M}$ is acted upon by $Hom( \pi_1(\Sigma) , \mathbb{Z}_2) \cong Hom( H_1(\Sigma ; \mathbb{Z}) , \mathbb{Z}_2)$. So we can change the framing $\xi_\Sigma$, by a homomorphism $\phi : H_1(\Sigma ; \mathbb{Z}) \to \mathbb{Z}_2$ such that $\phi[ S^1_1] = 0$, $\phi[S^1_2] = \phi[S^1_3] = 1$. This changes the framings on $\mathcal{M}_\Sigma |_{S^1_2}$ and $\mathcal{M}_\Sigma|_{S^1_3}$ but leaves the framing on $\mathcal{M}_\Sigma|_{S^1_1}$ unchanged. If $\xi''$ denotes the resulting framing on $\mathcal{M}_\Sigma$, then $\xi''|_{S^1_1} = \xi$, $\xi''|_{S^1_2} = \xi_0 \times \eta$ and $\xi''|_{S^1_3} = \xi'$. Here $\eta$ denotes the Lie group framing on $S^1_2$. Note that $\eta$ is the generator of $\Omega^{fr}_{1}$. Now $( \mathcal{M}_\Sigma , \xi'')$ is a framed cobordism from $(\mathcal{M} , \xi) \cup ( \mathcal{M}_0 \times S^1 , \xi_0 \times \eta)$ to $(\mathcal{M} , \xi')$, hence we have an equality
\[
[\mathcal{M}(\xi)] + \eta \cdot BF_{X,\mathfrak{s}} = [\mathcal{M}(\xi')].
\]

\end{proof}

As a consequence of Proposition \ref{prop:xixi'}, if the ordinary Bauer--Furuta invariant $BF_{X,\mathfrak{s}}$ is zero (or more generally, if $\eta \cdot BF_{X,\mathfrak{s}} = 0$), then the framed cobordism class of $\mathcal{M}$ does not depend on the choice of trivialisations and so defines an invariant of the family $(E , \mathfrak{s})$. We denote this invariant by $FBF(E , \mathfrak{s})$ and call it the {\em families Bauer--Furuta invariant of $(E , \mathfrak{s})$}. In what follows we will re-interpret this as an invariant of diffeomorphisms of $X$ via the mapping torus construction.

\subsection{A families Bauer--Furuta invariant of diffeomorphisms}\label{sec:fbf}

Let $f$ be an orientation preserving diffeomorphism of $X$. Define the mapping torus $E(f)$ of $f$ to be $X \times [0,1]/\! \sim$ where $(x,0) \sim (f(x),1)$. Let $\pi : E(f) \to S^1 = \mathbb{R}/\mathbb{Z}$ be the map $\pi(x,t) = t$. Then $\pi : E(f) \to S^1$ is a smooth fibre  bundle over $S^1$ with fibres diffeomorphic to $X$.

\begin{lemma}\label{lem:famspinc}
There is a bijection between families spin$^c$-structures on $E(f)$ and $f$-invariant spin$^c$-structures on $X$. For any $t \in S^1$, let $i_t : X \to E(f)$ be the inclusion $i_t(x) = (x,t)$. Then the bijection is given by the restriction map $\mathfrak{s} \mapsto \mathfrak{s}|_{X_t} = i_t^* \mathfrak{s}$.
\end{lemma}
\begin{proof}
Without loss of generality it suffices to show that the restriction map $\mathfrak{s} \mapsto \mathfrak{s}|_{X_0}$ is a bijection. For injectivity, let $\mathfrak{s},\mathfrak{s}'$ be two spin$^c$-structures such that $\mathfrak{s}|_{X_0} \cong \mathfrak{s}'|_{X_0}$. Then $\mathfrak{s}' \cong L \otimes \mathfrak{s}$ for some line bundle $L$ such that $L|_{X_0} \cong \mathbb{C}$. Since $b_1(X) = 0$, the Serre spectral sequence applied to $E(f) \to S^1$ implies that $H^2(E ; \mathbb{Z}) \cong H^2(X ; \mathbb{Z})^f = \{ \alpha \in H^2(X ; \mathbb{Z}) \; | \; f^*(\alpha) = \alpha \}$. In particular, the restriction map $H^2(E ; \mathbb{Z}) \to H^2( X_0 ; \mathbb{Z})$ is injective. Hence $L \cong \mathbb{C}$ and $\mathfrak{s} \cong \mathfrak{s}'$.

It remains to show surjectivity. Let $\mathfrak{s}$ be an $f$-invariant spin$^c$-structure on $X$. Then $\mathfrak{s}$ pulls back to a spin$^c$-structure $\widehat{\mathfrak{s}}$ on the vertical tangent bundle of the trivial family $\widehat{E} = X \times \mathbb{R} \to \mathbb{R}$. Let $\widehat{f} : X \times \mathbb{R} \to X \times \mathbb{R}$ be given by $\widehat{f}(x,t) = (f(x) , t+1)$. Then since $f^*(\mathfrak{s}) \cong \mathfrak{s}$, it follows that $\widehat{f}^*(\widehat{\mathfrak{s}}) \cong \widehat{\mathfrak{s}}$. Then by choosing a lift of $\widehat{f}$ to the spinor bundles of $\widehat{\mathfrak{s}}$, we can then descend $\widehat{\mathfrak{s}}$ to the quotient $\widehat{E}/\langle \widehat{f} \rangle = E(f)$.
\end{proof}

In light of Lemma \ref{lem:famspinc} we will identify families spin$^c$-structrues on $E(f)$ with $f$-invariant spin$^c$-structures on $X$ without further mention.

Given an orientation preserving diffeomorphism $f$ of $X$ and an $f$-invariant spin$^c$-structure $\mathfrak{s}$, we obtain a $1$-parameter family $E(f) \to S^1$ and $\mathfrak{s}$ determines a families spin$^c$-structure on $E(f)$, which we continue to denote by $\mathfrak{s}$. The bundle $H^+(E(f)/S^1)$ is trivial or non-trivial according to its first Stiefel--Whitney class. The space of oriented, maximal positive definite subspaces of $H^2(X ; \mathbb{R})$ has two connected components. Define $sgn_+(f) = \pm 1$ according to whether $f$ preserves or swaps the two components. Then it is easily seen that $H^+(E(f)/S^1)$ is trivial if and only if $sgn_+(f) = 1$. Assuming that $sgn_+(f) = 1$ and that $BF_{X,\mathfrak{s}} = 0$, we can then take the families Bauer--Furuta invariant of $(E(f) , \mathfrak{s})$ which is then an invariant of the triple $(X , \mathfrak{s} , f)$

\begin{definition}
Let $X$ be a compact, oriented, smooth $4$-manifold with $b_1(X) = 0$, $f : X \to X$ an orientation preserving diffeomorphism such that $sgn_+(f) = 1$ and $\mathfrak{s}$ an $f$-invariant spin$^c$-structure. Assume further that $BF_{X , \mathfrak{s}} = 0$ (or more generally, that $\eta \cdot BF_{X , \mathfrak{s}} = 0$). Then the {\em (non-equivariant) families Bauer--Furuta invariant of $(X,\mathfrak{s},f)$}, denoted 
\[
FBF_{X , \mathfrak{s}}(f) \in \Omega^{fr}_{d(\mathfrak{s}) + 2}
\]
is the families Bauer--Furuta invariant of $(E(f) , \mathfrak{s})$.
\end{definition}

Notice that $FBF_{X , \mathfrak{s}}(f)$ depends only on the smooth isotopy class of $f$. Indeed, to a smooth isotopy $f_t$ we can construct a smooth family over $[0,1] \times S^1$ whose restriction to $\{t\} \times S^1$ is $E(f_t)$. In turn, this yields a homotopy of Bauer--Furuta maps associated to each of the families $E(f_t)$.

\begin{proposition}\label{prop:compact1}
Let $X$ be a compact, oriented, smooth $4$-manifold with $b_1(X)=0$, $f : X \to X$ an orientation preserving diffeomorphism such that $sgn_+(f) = 1$. Then there are only finitely many spin$^c$-structures on $X$ such that $FBF_{X,\mathfrak{s}}(f)$ is defined and non-zero.
\end{proposition}
\begin{proof}
Let $E(f)$ be the mapping torus of $f$. A family of fibrewise metrics on $E(f)$ can be thought of as a path $\{ g_t \}_{t \in [0,1]}$ of Riemannian metrics on $X$ such that $f^*(g_1) = g_0$. Compactness properties of the Seiberg--Witten equations implies that there are only finitely many spin$^c$-structures $\mathfrak{s}$ on $X$ such that the Seiberg--Witten equations for $(X , \mathfrak{s} , g_t)$ have a solution for some $t \in [0,1]$ and with the zero perturbation. Let $S = \{ \mathfrak{s}_1, \dots , \mathfrak{s}_r \}$ be the set of spin$^c$-structures for which there is a solution. Let $\mathfrak{s}$ be a spin$^c$-structure such that $FBF_{X,\mathfrak{s}}(f)$ is defined (so $f^*(\mathfrak{s}) = \mathfrak{s}$ and $\eta \cdot BF_{X , \mathfrak{s}} = 0$). Let $SW : \mathbb{W} \to \mathbb{W}'$ be the corresponding families Seiberg--Witten monopole map. If $\mathfrak{s}$ does not belong to $S$, then the families moduli space $SW^{-1}(0)$ is empty. Let $h : S_B^W \to S_B^{W'}$ denote a finite dimensional approximation of $SW$. From the proof of \cite[Theorem 2.24]{bk}, $W,W'$ can be chosen in such a way that there exists a map $h' : S_B^W \to S_B^{W'}$ which is homotopic to $h$ and a section $\mu' : B \to S_B^{W'}$ such that $(h')^{-1}(\mu')$ and $SW^{-1}(0)$ are diffeomorphic (the map we are calling $h'$ corresponds to the map which was called $\widetilde{f}_2$ in \cite[\textsection 2]{bk}). In particular, since $SW^{-1}(0)$ is empty, so is $(h')^{-1}(\mu')$. It follows that the moduli space $\mathcal{M} = h^{-1}(\mu)$ for $h$ represents the trivial framed-cobordism class, hence $FBF_{X, \mathfrak{s}}(f) = 0$.
\end{proof}

For a compact, oriented smooth $4$-manifold $X$, let $Diff(X)$ denote the group of orientation preserving diffeomorphisms of $X$, with the $\mathcal{C}^\infty$-topology. Define the mapping class group of $X$ to be $M(X) = \pi_0(Diff(X))$, the group of isotopy classes of diffeomorphisms. Define $M_+(X)$ to be the subgroup of $M(X)$ consisting of isotopy classes of diffeomorphisms satisfying $sgn_+(f) = 1$. If $\mathfrak{s}$ is a spin$^c$-structure on $X$, let $M(X,\mathfrak{s})$ be the subgroup of $M(X)$ consisting of isotopy classes of diffeomorphisms preserving the isomorphism class of $\mathfrak{s}$. We also set $M_+(X,\mathfrak{s}) = M_+(X) \cap M(X , \mathfrak{s})$. Define the Torelli group $T(X)$ to be the subgroup of $M(X)$ consisting of isotopy classes of diffeomorphisms which are continuously isotopic to the identity. When $X$ is simply-connected, $T(X)$ is also the subgroup of $M(X)$ consitsing of those isotopy classes which act trivially on $H^2(X ; \mathbb{Z})$ \cite{qui}.

Since $FBF_{X,\mathfrak{s}}(f)$ depends only on the isotopy class of $f$, we can regard $FBF_{X,\mathfrak{s}}$ as a map
\[
FBF_{X , \mathfrak{s}} : M_+(X , \mathfrak{s}) \to \Omega^{fr}_{d(\mathfrak{s})+2}.
\]

\begin{proposition}\label{prop:hom}
The map $FBF_{X,\mathfrak{s}} : M_+(X , \mathfrak{s}) \to \mathbb{Z}_2$ is a group homomorphism.
\end{proposition}
\begin{proof}
Let $f_1,f_2$ be orientation preserving diffeomorphisms of $X$ preserving $\mathfrak{s}$ and with $sgn_+(f_1) = sgn_+(f_2) = 1$. Let $\Sigma$ denote a pair of pants, that is, the $2$-sphere with three open balls removed. We regard $\Sigma$ as a cobordism with ingoing boundary $\partial_- \Sigma = S^1$ and outgoing boundary $\partial_+ \Sigma = S^1 \coprod S^1$. Since $\Sigma$ deformation retracts to $S^1 \vee S^1$, we can construct a smooth fibre bundle $E \to \Sigma$ with fibres diffeomorphic to $X$ whose restriction to $\partial_- \Sigma$ is the mapping cylinder $E(f_1 \circ f_2 )$ and whose restriction to $\partial_+ \Sigma$ is the disjoint union $E(f_1) \coprod E(f_2)$ of mapping cylinders of $f_1$ and $f_2$. Furthermore, the spin$^c$-structure $\mathfrak{s}$ can be promoted to a families spin$^c$-structure on $E$ that restricts to $\mathfrak{s}$ on each fibre. We denote this families spin$^c$-structure simply by $\mathfrak{s}$.

Associated to the family $(E , \mathfrak{s})$ is a families Bauer--Furuta map $f : S^{W}_\Sigma \to S^{W'}_{\Sigma'}$. Since $ind(D), H^+(E/\Sigma)$ and $T\Sigma$ are all trivialisable, we can follow the procedure given in Section \ref{sec:moduli} to construct a stable framing of the moduli space $\mathcal{M} = f^{-1}(\mu)$, where $\mu$ is a section of $S^{W'}_{\Sigma'}$ which is transverse to $f$ and disjoint from infinity. Since $\Sigma$ has a boundary $\mathcal{M}$ will also have a boundary, which is given by $\mathcal{M}|_{\partial \Sigma}$. Clearly $\mathcal{M}|_{\partial_- \Sigma}$ is a representative of the framed cobordism class $FBF_{X , \mathfrak{s}}(f_1 \circ f_2)$ and similarly $\mathcal{M}|_{\partial_+ \Sigma}$ represents the framed cobordism class $FBF_{X , \mathfrak{s}}(f_1) + FBF_{X,\mathfrak{s}}(f_2)$. Then since $\mathcal{M}$ is a stably framed cobordism from $\mathcal{M}_{\partial_- \Sigma}$ to $\mathcal{M}_{\partial_+ \Sigma}$, we get the equality $FBF_{X,\mathfrak{s}}(f_1 \circ f_2) = FBF_{X,\mathfrak{s}}(f_1) + FBF_{X,\mathfrak{s}}(f_2)$.
\end{proof}

The next result is a simple consequence of diffeomorphism invariance of the Bauer--Furuta invariant. We let $Diff(X)$ act on the set of spin$^c$-structures by inverse pullback: $f \cdot \mathfrak{s} = (f^{-1})^*(\mathfrak{s})$ (inverting makes this a left action).

\begin{proposition}\label{prop:diffinv}
Let $X$ be a compact, oriented, smooth $4$-manifold with $b_1(X) = 0$, $f : X \to X$ an orientation preserving diffeomorphism such that $sgn_+(f) = 1$ and $\mathfrak{s}$ an $f$-invariant spin$^c$-structure. Assume further that $BF_{X , \mathfrak{s}} = 0$ (or more generally $\eta \cdot BF_{X , \mathfrak{s}} = 0$). Let $h : X \to X$ be any orientation preserving diffeomorphism. Then
\[
FBF_{X , h \cdot \mathfrak{s} }( h \circ f \circ h^{-1} ) = sgn_+(h) FBF_{X , \mathfrak{s}}(f).
\]
\end{proposition}
\begin{proof}
The diffeomorphism $h$ defines an isomorphism of mapping tori $H : E(f) \to E(h \circ f \circ h^{-1})$ given by $H(x,t) = ( h(x) , t)$. Under this isomorphism the families spin$^c$-structure on $E(f)$ corresponding to $\mathfrak{s}$ is sent to the families spin$^c$-structure on $E(h \circ f \circ h^{-1})$ corresponding to $h \cdot \mathfrak{s}$. It follows that the families Bauer--Furuta maps for $(E(f) , \mathfrak{s})$ and $( E(h \circ f \circ h^{-1}) , h \cdot \mathfrak{s})$ can be identified. However this identification will either preserve or reverse orientation on $H^+(X)$ according to $sgn_+(h)$, hence the equality $FBF_{X , h \cdot \mathfrak{s} }( h \circ f \circ h^{-1} ) = sgn_+(h) FBF_{X , \mathfrak{s}}(f)$.
\end{proof}

We will prove two gluing formulas for the families Bauer--Furuta invariant $FBF_{X,\mathfrak{s}}$. For the first gluing formula we assume that $X = X' \# (S^2 \times S^2)$, for some compact, oriented, smooth $4$-manifold $X'$ with $b_1(X') = 0$. Assume that $\mathfrak{s} = \mathfrak{s}' \# \mathfrak{s}_0$, where and $\mathfrak{s}_0$ is the unique spin$^c$-structure on $S^2 \times S^2$ with $c(\mathfrak{s}_0) = 0$. Consider a diffeomorphism $f$ of $X$ which can be written as a connected sum $f = f' \# \rho$. More precisely, this means that $f'$ is diffeomorphism on $X'$, $\rho$ is a diffeomorphism on $S^2 \times S^2$ and $f',\rho$ both equal the identity in some neighbourhood of the points where the connected sum is formed. Then $f'$ and $\rho$ can be glued together to give a diffeomorphism $f = f' \# \rho$ on $X$. We will assume that $sgn_+(f') = sgn_+(\rho) = -1$. We will also assume that $\rho$ is isotopic to an isometry of the standard positive scalar curvature metric on $S^2 \times S^2$.

\begin{lemma}\label{lem:psc}
Let $\pi : E \to B$ be a smooth family of compact, oriented spin $4$-manifolds diffeomorphic to $X$, where $b_1(X) = 0$. Let $\mathfrak{s}$ be a families spin$^c$-structure which comes from a families spin structure (i.e. a spin structure on the vertical tangent bundle). Suppose further that $E$ can be given a family of fibrewise metrics $g = \{ g_b \}_{b \in B}$ such that each $g_b$ has positive scalar curvature. Then the families Bauer--Furuta invariant of $(E , \mathfrak{s})$ is given by the inclusion map $i : S^0_B \to S^{H^+(E/B)}_B$.
\end{lemma}
\begin{proof}
The idea of the proof is that in the presence of positive scalar curvature we can homotopy away the quadratic terms in the Seiberg--Witten equations. Since the families spin$^c$-structure comes from a spin structure we can choose the reference connection $A$ to be a spin connection, so $F_A = 0$. On each fibre of $E$ the Seiberg--Witten monopole map takes the form
\begin{align*}
SW : L^2_k( S^+ ) \oplus L^2_k( \wedge^1 T^*X ) \oplus &\to L^2_{k-1}(S^-) \oplus L^2_{k-1}( \wedge^2_+ T^*X) \oplus L^2_{k-1}( \wedge^0 T^*X )_0 \\
SW( \psi , a ) &= ( D_{A+ia} \psi , d^+ a + i\sigma(\psi) , d^*a).
\end{align*}

Following \cite{bf}, we say that a Fredholm map $f : H' \to H$ of separable Hilbert spaces is {\em bounded} if the preimage of bounded sets are bounded. This notion also makes sense for Fredholm maps of Hilbert bundles over a base space $B$. Consider the homotopy $H_t$ where
\[
H_t( \psi , a ) = ( D_{A+tia} \psi , d^+a + ti\sigma(\psi) , d^*a).
\]
We have $H_t = SW$ and $H_0$ is the linear Fredholm map $H_0(\psi , a) = (D_A \psi , d^+a , d^*a)$. At times $t = 0,1$, $H_t$ is a bounded map. For $t=1$, this is due to Bauer--Furuta \cite[Proposition 3.1]{bf}. For $t=0$, this follows from $Ker(D_A) = 0$, which holds because of positive scalar curvature and the Weitzenb\"ock formula.

We will show that for all times $t \in [0,1]$, the solution set $H_t^{-1}(0)$ is bounded, in fact, consists of a single point for each fibre of $E \to B$. As explained in \cite[\textsection 3]{b2}, this implies that the finite-dimensional approximations at $t=0,1$ give the same stable cohomotopy class.

Suppose $(\psi,a)$ is a solution to $H_t( \psi , a) = 0$ for some $t \in [0,1]$ lying over $b \in B$. Thus $D_{A+tia} \psi = 0$, $d^+ a = -t i\sigma(\phi)$ and $d^*a = 0$.

From the Weitzenb\"ock formula we get ($s$ denotes the scalar curvature):
\begin{align*}
\Delta |\psi|^2 &\le 2 \langle \nabla_{A+tia}^* \nabla_{A+tia} \psi , \psi \rangle \\
& = \langle 2 D^*_{A+tia}D_{A+tia}\psi -\frac{s}{2}\psi - F^+_{A+tia} \psi , \psi \rangle \\
& = -\frac{s}{2}|\psi|^2 - t \langle i (d^+a) \psi , \psi \rangle \\
& = -\frac{s}{2}|\psi|^2 - t^2 \langle \sigma(\psi) \psi , \psi \rangle \\
&= -\frac{s}{2}|\psi|^2 - \frac{t^2}{2} |\psi|^4. 
\end{align*}
Hence
\[
\Delta |\psi|^2 +\frac{s}{2}|\psi|^2 + \frac{t^2}{2}|\psi|^4 \le 0.
\]

At the global maximum $p \in \pi^{-1}(b)$ of $|\psi|^2$, we have $|\psi| = ||\psi||_{L^\infty}$ and $\Delta |\psi|^2 \ge 0$, hence
\[
\frac{s(p)}{2}||\psi||^2_{L^\infty} + \frac{t^2}{2}||\psi||^4_{L^\infty} \le 0.
\]

Since $s(p) > 0$, this means $\psi = 0$. So we now have $d^+a = 0$, $d^*a = 0$. Hence $a=0$ (as $b_1(X)=0$) and we have a unique reducible solution given by the connection $A$. Of course this means that $H_t^{-1}(0)$ is bounded.

Since $H_0 = D_A \oplus (d^+) \oplus d^*$ is linear Fredholm with trivial kernel and with cokernel $H^+(E/B)$, it is easily seen that the finite dimensional approximation of $H_0$ is the inclusion map $ i : S^0_B \to S^{H^+(E/B)}_B$.
\end{proof}

\begin{proposition}\label{prop:gluing1}
Let $X = X' \# (S^2 \times S^2)$ for some compact, oriented, smooth $4$-manifold $X'$ with $b_1(X') = 0$. Assume that $\mathfrak{s} = \mathfrak{s}' \# \mathfrak{s}_0$, where $\mathfrak{s}_0$ is the unique spin$^c$-structure on $S^2 \times S^2$ with $c(\mathfrak{s}_0) = 0$. Let $f$ be a diffeomorphism of $X$ of the form $f = f' \# \rho$ where $sgn_+(f') = sgn_+(\rho) = -1$ and where $\rho$ is isotopic to an isometry of the standard positive scalar curvature metric on $S^2 \times S^2$. Then
\[
FBF_{X , \mathfrak{s}}(f) = BF_{X' , \mathfrak{s}'} \in \Omega^{fr}_{d(\mathfrak{s})+2} = \Omega^{fr}_{d(\mathfrak{s}')+1}.
\]
\end{proposition}
\begin{proof}
First note that since $X = X' \# (S^2 \times S^2)$, the ordinary Bauer--Furuta invariant of $X$ vanishes, so $FBF_{X,\mathfrak{s}}(f)$ is well-defined. By the gluing formula for the families Bauer--Furuta map \cite{tom1}, the families Bauer--Furuta map for $(X,\mathfrak{s},f)$ will be of the form $F_1 \wedge_{S^1} F_2$, where $F_1$ is the families Bauer--Furuta map for $(X',\mathfrak{s'}, f')$ and $F_2$ is the famlies Bauer--Furuta map for $( (S^2 \times S^2) , \mathfrak{s}_0 , \rho)$. Since $\rho$ is isotopic to an isometry $\rho'$ of the standard positive scalar curvature metric on $S^2 \times S^2$, the mapping torus of $\rho$ is diffeomorphic to the mapping torus of $\rho'$ and hence the mapping torus of $\rho'$ can be equipped with a family of positive scalar curvature metrics on the fibres. Lemma \ref{lem:psc} then implies that $F_2$ is given by the inclusion $i : S^0_B \to S^{M}_B$, where $B = S^1$ and $M \to S^1$ is the unique non-orientable line bundle on $S^1$ (the M\"obius band). Choosing a perturbation $\mu : B \to S^M_B$ that crosses the zero section exactly once transversally, we see that the moduli space for $F_1 \wedge_{S^1} F_2$ is equal to the unparametrised moduli space for $(X' , \mathfrak{s}')$. The equality $FBF_{X , \mathfrak{s}}(f) = BF_{X' , \mathfrak{s}'}$ follows.
\end{proof}

We now prove a second gluing formula.

\begin{proposition}\label{prop:gluing2}
Let $X$ be a compact, oriented smooth $4$-manifold with $b_1(X) = 0$. Let $f : X \to X$ be an orientation preserving diffeomorphism with $sgn_+(f) = 1$ and let $\mathfrak{s}$ be a spin$^c$-structure on $X$ which is preserved by $f$. Assume that $BF_{X , \mathfrak{s}} = 0$ (or just $\eta \cdot BF_{X , \mathfrak{s}} = 0$), so that $FBF_{X,\mathfrak{s}}(f)$ is defined.

Let $W$ be a compact, oriented smooth $4$-manifold with $b_1(W) = 0$ and let $\mathfrak{s}_W$ be a spin$^c$-structure on $W$. Assume that $f$ fixes a neighbourhood of a point on $X$ so that the connected sum diffeomorphism $f \# id_W$ can be constructed. Then
\[
FBF_{X \# W , \mathfrak{s} \# \mathfrak{s}_W}(f \# id_W) = FBF_{X , \mathfrak{s}}(f) \cdot BF_{W , \mathfrak{s}_W}.
\]
\end{proposition}
\begin{proof}
Once again, the gluing formula for the families Bauer--Furuta invariant \cite{tom1} implies that the families Bauer--Furuta invariant for $(X \# W , \mathfrak{s} \# \mathfrak{s}_W , f \# id_W)$ is of the form $F^1 \wedge_{S^1} F_2$, where $F_1$ is the families Bauer--Furuta map for $(X , \mathfrak{s} , f)$ and $F_2$ is the families Bauer--Furuta map for $(W , \mathfrak{s}_W , id_W)$. The latter is just the pullback under $S^1 \to \{pt\}$ of the ordinary Bauer--Furuta invariant of $(W , \mathfrak{s}_W)$. From here the result follows easily.
\end{proof}

Next we prove a blowup formula for the families Bauer--Furuta invariant. Assume that $X = X' \# \overline{\mathbb{CP}^2}$, where $X'$ is a compact, oriented, smooth $4$-manifold with $b_1(X') = 0$. Assume that $\mathfrak{s}$ is of the form $\mathfrak{s} = \mathfrak{s}' \# \kappa$, where $c(\kappa)^2 = -1$. We assume that $BF_{X',\mathfrak{s}'} = 0$, hence also $BF_{X,\mathfrak{s}} = 0$. Assume also that $f$ is of the form $f = f' \# id$ for some diffeomorphism $f'$ on $X'$ with $sgn_+(f') = 1$.

\begin{proposition}
Under the above assumptions on $(X , \mathfrak{s} , f)$, we have:
\[
FBF_{X , \mathfrak{s}}(f) = FBF_{X' , \mathfrak{s}'}(f') \in \Omega^{fr}_{d(\mathfrak{s}) + 2}.
\]
\end{proposition}
\begin{proof}
The gluing formula for the families Bauer--Furuta invariant gives that the families Bauer--Furuta invariant for $(X , \mathfrak{s} , f)$ is of the form $F = F_1 \wedge_{S^1} F_2$, where $F_1$ is the families Bauer--Furuta invariant of $(X' , \mathfrak{s}' , f')$ and $F_2$ is the families Bauer--Furuta invariant for $(\overline{\mathbb{CP}}^2 , \kappa , id_{\overline{\mathbb{CP}^2}} )$. Since we are taking the identity diffeomorphism $id_{\overline{\mathbb{CP}^2}}$, the families Bauer--Furuta invariant $F_2$ is just the pullback under $S^1 \to pt$ of the ordinary (non-equivariant) Bauer--Furuta invariant of $(\overline{\mathbb{CP}^2} , \kappa)$. But since $d(\kappa) = -1$, this takes the form of a homotopy class $F_2 : S^n \to S^n$, that is a class in $\pi^s_0 \cong \mathbb{Z}$, detected by the degree of $F_2$. But the degree of $F_2$ is easily seen to be $1$ (for example, using the method of \cite[\textsection 3]{bar}). Therefore $F$ and $F_1$ are stably homotopic and hence $FBF_{X , \mathfrak{s}}(f) = FBF_{X' , \mathfrak{s}'}(f')$.
\end{proof}

\subsection{$Pin$-cobordism-valued invariants}\label{sec:pin}

The families Bauer--Furuta invariant constructed in Section \ref{sec:fbf} has an unfortunate defect in that it is only defined for diffeomorphisms satisfying $sgn_+(f) = 1$. For diffeomorphisms with $sgn_+(f) = -1$, we do not get a naturally defined framing and so we can not expect to get invariants valued in $\Omega^{fr}_*$. Instead we consider invariants valued in other cobordism rings. Since the families moduli spaces with $sgn_+(f) = -1$ do not come with natural orientations, obvious choices such as oriented cobordism or spin cobordism are not suitable. Unoriented cobordism is also not particulaly useful for our purposes, since $\Omega^O_1 = 0$. It turns out that pin-cobordism $\Omega^{pin}_*$ leads to an invariant with satisfactory properties. Here and throughout this paper, pin-cobordism refers to $Pin^-$-cobordism in the sense of Kirby--Taylor \cite{kt}.

As in Section \ref{sec:moduli}, we start with the general setting of a smooth family $E \to B$ and a families spin$^c$-structure $\mathfrak{s}$. We obtain a Bauer--Furuta map $f : S^W_B \to S^{W'}_B$ and a moduli space $\mathcal{M} = f^{-1}(\mu)$. As in Section \ref{sec:moduli} we have a stable isomorphism
\[
T\mathcal{M} = TB + W - W' = TB + ind(D) - H^+(E/B) \text{ in } K^0(B).
\]

Suppose that $B = S^1$ is the circle. Every vector bundle on $S^1$ admits a pin structure, hence $TB, W$ and $W'$ all have pin structures. We will see that $T\mathcal{M}$ inherits a pin structure and hence we obtain a cobordism class $[\mathcal{M}] \in \Omega^{pin}_*$. To do this properly requires some care because in the case of pin structures, the order of summands in a direct sum matter. Recall that if $U \to B$ admits a pin structure, then the set of pin structures on $U$ is a torsor for $H^1(B ; \mathbb{Z}_2)$. If $U_1,U_2 \to B$ are vector bundles on $B$ with pin-structures $\mathfrak{p}_1, \mathfrak{p}_2$, then we get induced pin structures $\mathfrak{p}_1 \oplus \mathfrak{p}_2$ on $U_1 \oplus U_2$ and $\mathfrak{p}_2 \oplus \mathfrak{p}_1$ on $U_2 \oplus U_1$. Let $s : U_1 \oplus U_2 \to U_2 \oplus U_1$ be the swapping isomorphism $s(u_1,u_2) = (u_2 , u_1)$. One can check that 
\[
s^*( \mathfrak{p}_2 \oplus \mathfrak{p}_1) = det(U_1) \otimes det(U_2) \otimes (\mathfrak{p}_1 \oplus \mathfrak{p}_2).
\]

Recall from Secction \ref{sec:moduli}, the isomorphism
\[
\rho_{f} : TB \oplus W \oplus \mathbb{R} \to T \mathcal{M} \oplus W' \oplus \mathbb{R}.
\]

After suspending, we can assume $W'$ is trivial, say $W' \cong \mathbb{R}^m$ for some $m$. So we have an isomorphism
\[
T \mathcal{M} \oplus \mathbb{R}^m \cong TB \oplus W \oplus \mathbb{R}.
\]
Now $TB, W$ and $\mathbb{R}$ admit pin structures, hence so does $T \mathcal{M} \oplus \mathbb{R}^m$. The obstruction for a pin structure is $w_2 + w_1^2$, which is a stable characteristic class. Thus if $T\mathcal{M} \oplus \mathbb{R}^m$ admits a pin structure, then so does $T\mathcal{M}$. Choose pin structures $\mathfrak{p}_{TB},\mathfrak{p}_W, \mathfrak{p}_{\mathbb{R}}$ and $\mathfrak{p}_{\mathbb{R}^m}$ (all regarded as vector bundles over $S^1$). Then there is a uniquely determined pin structure $\mathfrak{p}_{\mathcal{M}}$ on $T\mathcal{M}$ such that $\mathfrak{p}_{\mathcal{M}} \oplus \mathfrak{p}_{\mathbb{R}^m} \cong \mathfrak{p}_{TB} \oplus \mathfrak{p}_W \oplus \mathfrak{p}_{\mathbb{R}}$. Since $H^1(S^1 ; \mathbb{Z}_2) \cong \mathbb{Z}_2$, any vector bundle on $S^1$ admits exactly two pin structures. It follows that changing any of the pin structures on $TB,W, \mathbb{R}$ or $\mathbb{R}^m$ changes $\mathfrak{p}_{\mathcal{M}}$ to $\mathfrak{p}'_{\mathcal{M}} = M \otimes \mathfrak{p}_{\mathcal{M}}$, where $M$ is the unique non-trivial real line bundle on $S^1$.

It it easily seen that stabilisation of $f$ does not change the pair of pin structures $\mathfrak{p}_{\mathcal{M}}, \mathfrak{p}'_{\mathcal{M}}$. Denote by $[\mathcal{M}], [\mathcal{M}'] \in \Omega^{pin}_{d(\mathfrak{s})+2}$ the resulting pair of pin cobordism classes. By essentially the same proof as that of Proposition \ref{prop:xixi'}, we obtain:

\begin{proposition}
We have
\[
[\mathcal{M}] - [\mathcal{M}'] = \eta ^{pin} \cdot BF_{X,\mathfrak{s}} \in \Omega^{pin}_{d(\mathfrak{s})+2}
\]
where $\eta^{pin}$ denotes the generator of $\Omega^{pin}_1 \cong \mathbb{Z}_2$ and we regard $BF_{X,\mathfrak{s}}$ as an element of $\Omega^{pin}_{d(\mathfrak{s})+1}$ by using the natural map $\Omega^{fr}_* \to \Omega^{pin}_*$.
\end{proposition}

In particular, if $BF_{X,\mathfrak{s}} = 0$ (or more generally if $\eta ^{pin} \cdot BF_{X,\mathfrak{s}} = 0 \in \Omega^{pin}_{d(\mathfrak{s})+2}$), then the pin cobordism class $[\mathcal{M}] \in \Omega^{pin}_{d(\mathfrak{s})+2}$ does not depend on the choices made in its construction, so gives a well-defined invariant of $(X,\mathfrak{s},f)$.

\begin{definition}
Let $X$ be a compact, oriented, smooth $4$-manifold with $b_1(X) = 0$, $f : X \to X$ an orientation preserving diffeomorphism and $\mathfrak{s}$ an $f$-invariant spin$^c$-structure. Assume further that $BF_{X , \mathfrak{s}} = 0$ (or more generally, that $\eta^{pin} \cdot BF_{X , \mathfrak{s}} = 0 \in \Omega^{pin}_{d(\mathfrak{s})+2}$). Then the {\em pin-cobordism valued families Bauer--Furuta invariant of $(X,\mathfrak{s},f)$}, denoted 
\[
FBF^{pin}_{X , \mathfrak{s}}(f) \in \Omega^{pin}_{d(\mathfrak{s}) + 2}
\]
is the pin-cobordism class $[\mathcal{M}]$ of the families moduli space $\mathcal{M}$ associated to the family $(E(f) , \mathfrak{s})$.
\end{definition}

Of course, if $sgn_+(f) = 1$, then $FBF^{pin}_{X,\mathfrak{s}}(f)$ is the image of $FBF_{X,\mathfrak{s}}(f)$ under the natural map $\Omega^{fr}_* \to \Omega^{pin}_*$.

Similar to $FBF_{X , \mathfrak{s}}(f)$, we have that $FBF^{pin}_{X,\mathfrak{s}}(f)$ depends only on the smooth isotopy class of $f$. Therefore, we can regard $FBF^{pin}_{X, \mathfrak{s}}$ as a map
\[
FBF^{pin}_{X , \mathfrak{s}} : M(X , \mathfrak{s}) \to \Omega^{pin}_{d(\mathfrak{s})+2}.
\]
The proof of Proposition \ref{prop:hom} easily adapts to the pin-cobordism case and gives that $FBF^{pin}_{X, \mathfrak{s}}$ is a group homomorphism.

\begin{remark}
Not every class in $\Omega^{pin}_{*}$ can be realised as the pin-cobordism valued families Bauer--Furuta invariant for some triple $(X,\mathfrak{s} , f)$. Indeed, if $sgn_+(f) = 1$, then $FBF^{pin}_{X,\mathfrak{s}}(f)$ lies in the image of $\Omega^{fr}_* \to \Omega^{pin}_*$ and is therefore a $2$-torsion class \cite[Corollary 1.13]{kt}. More generally, for any diffeomorphism $f$ that preserves $\mathfrak{s}$, we have that $sgn_+(f^2) = 1$ and so $FBF^{pin}_{X,\mathfrak{s}}(f^2) = 2FBF_{X,\mathfrak{s}}^{pin}(f)$ lies in the image of $\Omega^{fr}_* \to \Omega^{pin}_*$. In particular, $FBF_{X , \mathfrak{s}}(f)$ is always $4$-torsion. On the other hand $\Omega^{pin}_2 \cong \mathbb{Z}_8$, so not every class in $\Omega^{pin}_2$ can be realised.
\end{remark}

The proof of Proposition \ref{prop:compact2} easily adapts to the case of pin-cobordism, giving:

\begin{proposition}\label{prop:compact2}
Let $X$ be a compact, oriented, smooth $4$-manifold wth $b_1(X) = 0$, $f : X \to X$ an orientation preserving diffeomorphism. Then there are only finitely many spin$^c$-structures on $X$ such that $FBF^{pin}_{X,\mathfrak{s}}(f)$ is defined and non-zero.
\end{proposition}

The gluing and blowup formulas are easily be adapted to the pin-cobordism case, giving the following results.

\begin{proposition}\label{prop:gluing1pin}
Let $X = X' \# (S^2 \times S^2)$ for some compact, oriented, smooth $4$-manifold $X'$ with $b_1(X') = 0$. Assume that $\mathfrak{s} = \mathfrak{s}' \# \mathfrak{s}_0$, where $\mathfrak{s}_0$ is the unique spin$^c$-structure on $S^2 \times S^2$ with $c(\mathfrak{s}_0) = 0$. Let $f$ be a diffeomorphism of $X$ of the form $f = f' \# \rho$ where $sgn_+(\rho) = -1$ and where $\rho$ is isotopic to an isometry of the standard positive scalar curvature metric on $S^2 \times S^2$. Then
\[
FBF^{pin}_{X , \mathfrak{s}}(f) = BF^{pin}_{X' , \mathfrak{s}'} \in \Omega^{pin}_{d(\mathfrak{s})+2} = \Omega^{pin}_{d(\mathfrak{s}')+1}
\]
where $BF^{pin}_{X' , \mathfrak{s}'}$ denotes the image of $BF_{X',\mathfrak{s}'}$ in $\Omega^{pin}_*$.
\end{proposition}

\begin{proposition}\label{prop:gluing2pin}
Let $X$ be a compact, oriented smooth $4$-manifold with $b_1(X) = 0$. Let $f : X \to X$ be an orientation preserving diffeomorphism and let $\mathfrak{s}$ be a spin$^c$-structure on $X$ which is preserved by $f$. Assume that $BF^{pin}_{X , \mathfrak{s}} = 0$ (or just $\eta^{pin} \cdot BF^{pin}_{X , \mathfrak{s}} = 0$), so that $FBF^{pin}_{X,\mathfrak{s}}(f)$ is defined.

Let $W$ be a compact, oriented smooth $4$-manifold with $b_1(W) = 0$ and let $\mathfrak{s}_W$ be a spin$^c$-structure on $W$. Assume that $f$ fixes a neighbourhood of a point on $X$ so that the connected sum diffeomorphism $f \# id_W$ can be constructed. Then
\[
FBF^{pin}_{X \# W , \mathfrak{s} \# \mathfrak{s}_W}(f \# id_W) = FBF^{pin}_{X , \mathfrak{s}}(f) \cdot BF^{pin}_{W , \mathfrak{s}_W}.
\]
\end{proposition}

\begin{proposition}
Let $X = X' \# \overline{\mathbb{CP}^2}$, where $X'$ is a compact, oriented, smooth $4$-manifold with $b_1(X') = 0$. Assume that $\mathfrak{s}$ is of the form $\mathfrak{s} = \mathfrak{s}' \# \kappa$, where $c(\kappa)^2 = -1$. Assume that $BF^{pin}_{X',\mathfrak{s}'} = 0$, hence also $BF^{pin}_{X,\mathfrak{s}} = 0$. Assume that $f$ is of the form $f = f' \# id$ for some diffeomorphism $f'$ on $X'$. Then:
\[
FBF^{pin}_{X , \mathfrak{s}}(f) = FBF^{pin}_{X' , \mathfrak{s}'}(f') \in \Omega^{pin}_{d(\mathfrak{s}) + 2}.
\]
\end{proposition}

\section{Exotic diffeomorphisms}\label{sec:exoticdiff}

In this section, we will use the gluing and blowup formulas to construct exotic diffeomorphisms of reducible, simply-connected $4$-manifolds with odd $b_+$. 

\begin{theorem}\label{thm:exotic1}
Let $X$ be one of:
\begin{itemize}
\item[(1)]{$n (S^2 \times S^2) \# n K3$, $n \ge 1$, or}
\item[(2)]{$4 n \mathbb{CP}^2 \# k \overline{\mathbb{CP}^2}$, $n\ge 1$, $k \ge 20n$.}
\end{itemize}

Then the image of the homomorphism
\[
\Phi : T(X) \to \bigoplus_{\mathfrak{s}} \Omega^{fr}_1
\]
(where the sum is over all spin$^c$-structures on $X$ with $d(\mathfrak{s}) = -1$) given by $\Phi(f) = \bigoplus_{\mathfrak{s}} FBF_{X,\mathfrak{s}}(f)$ is not finitely generated.

\end{theorem}
\begin{proof}
First note that $\Phi$ is well-defined because of Proposition \ref{prop:compact1} and the fact that $BF_{X,\mathfrak{s}} = 0$ for every spin$^c$-structure on $X$. To show that the image is not finitely generated, it suffices to show that there are infinitely many spin$^c$-structures on $X$ with $d(\mathfrak{s}) = -1$ and $FBF_{X , \mathfrak{s}}(t) \neq 0$ for some $t \in T(X)$. In case (1), for each odd integer $m \ge 1$ we will contruct distinct spin$^c$-structures $\mathfrak{s}_m$ on $X$ with $d(\mathfrak{s}_m) = -1$ and isotopy classes $[t_m] \in T(X)$ such that $FBF_{X , \mathfrak{s}_m}(t_m) \neq 0$. In case (2) we will do the same, except that $m \ge 1$ will be an even integer.

For $m,n \ge 1$, let $E(2n)_m$ be the elliptic surface obtained from $E(2n)$ by performing a logarithmic transformation of multiplicity $m$. Then $E(2m)_m$ is spin if and only if $m$ is odd. Set $X_0 = (n-1)(S^2 \times S^2) \# nK3$ in case (1) and $X_0 = (4n-1)\mathbb{CP}^2 \# (k-1)\overline{\mathbb{CP}^2}$ in case (2). Then $X = X_0 \# (S^2 \times S^2)$. Let $L = H^2(X ; \mathbb{Z})$ be the intersection lattice of $X$, $L_0 = H^2(X ; \mathbb{Z})$ the intersection lattice of $X_0$ and $H = H^2(S^2 \times S^2 ; \mathbb{Z})$ the intersection lattice of $S^2 \times S^2$ so that $L = L_0 \oplus H$.

Let $r : S^2 \to S^2$ be a reflection and set $\rho' = r \times r : S^2 \times S^2 \to S^2 \times S^2$ so that $sgn_+(\rho') = -1$. Let $\rho$ be obtained from $\rho'$ by performing an isotopy so that $\rho$ fixes a neighbourhood of a point. Let $f_0 : X \to X$ be a diffeomorphism of the form $f_0 = id_{X_0} \# \rho$. 

Consider case (1) first. Let $X_m = E(2n)_m$ where $m \ge 1$ is odd. Then $E(2n)_m$ is spin and $X_m \# (S^2 \times S^2)$ is diffeomorphic to $X$ \cite{gom,man,moi}. By \cite[Theorem 2]{wall}, we may choose the diffeomorphism $\psi_m : X_m \# (S^2 \times S^2)$ so that induced isomorphism $H^2(X_m \# (S^2 \times S^2) ; \mathbb{Z}) \to H^2(X ; \mathbb{Z}) = L_0 \oplus H$ sends $H^2(X_m ; \mathbb{Z})$ to $L_0$ and sends $H^2(S^2 \times S^2 \; \mathbb{Z})$ to $H$. Let $f_m : X \to X$ be given by $f_m = \psi_m \circ (id_{X_m} \# \rho ) \circ \psi_m^{-1}$. Then $f_0$ and $f_m$ induce the same action on $L$ and hence $t_m = f_m \circ f_0$ defines a class $[t_m] \in T(X)$.

Let $\mathfrak{s}_m = (\psi_m^{-1})^*( \mathfrak{s}_{can} \# \mathfrak{s}_0)$, where $\mathfrak{s}_{can}$ is the canonical spin$^c$-structure on $E(2n)_m$ and $\mathfrak{s}_0$ is the unique spin structure on $S^2 \times S^2$ with $c(\mathfrak{s}_0) = 0$. Recall that the canonical class of $E(2n)_m$ is $2nm-m-1 = (2n-1)m - 1$ times a primitive class. Hence the same is true of $c(\mathfrak{s}_m)$ and so the spin$^c$-structures $\{ \mathfrak{s}_m \}$ are all distinct. We have $d(\mathfrak{s}_m) = -1$, so $FBF_{X , \mathfrak{s}_m}(t_m) \in \Omega_1^{fr}$. We show that $FBF_{X , \mathfrak{s}_m}(t_m) \neq 0$. Since $\Omega_1^{fr} \to \Omega_1^{pin}$ is an isomorphism, it suffices to show $FBF^{pin}_{X , \mathfrak{s}_m}(t_m) \neq 0$. We have:
\begin{align*}
FBF^{pin}_{X , \mathfrak{s}_m}(t_m) &= FBF^{pin}_{X,\mathfrak{s}_m}( f_m \circ f_0) \\
&= FBF^{pin}_{X,\mathfrak{s}_m}(f_m) + FBF^{pin}_{X , \mathfrak{s}_m}(f_0) \\
&= SW( E(2n)_m , \mathfrak{s}_{can}) \eta^{pin} \\
&= \eta^{pin}
\end{align*}
where in the second to last line we used the gluing formula. This proves case (1) of the theorem.

The proof in case (2) is very similar. In this case, for each even $m \ge 1$ we take $X_m = E(2n)_m \# (20n-k)\overline{\mathbb{CP}^2}$ and again we choose diffeomorphisms $\psi_m : X_m \# (S^2 \times S^2) \to X_0 \# (S^2 \times S^2)$ for which the induced isomorphism $H^2(X_m \# (S^2 \times S^2) ; \mathbb{Z}) \to H^2(X ; \mathbb{Z}) = L_0 \oplus H$ sends $H^2(X_m ; \mathbb{Z})$ to $L_0$ and sends $H^2(S^2 \times S^2 \; \mathbb{Z})$ to $H$. Let $f_m : X \to X$ again be given by $f_m = \psi_m \circ (id_{X_m} \# \rho ) \circ \psi_m^{-1}$ and set $t_m = f_m \circ f_0$. Let $\mathfrak{s}_m = (\psi_m^{-1})^*( \mathfrak{s}_{can} \# \mathfrak{s}_0)$, where $\mathfrak{s}_{can}$ is the canonical spin$^c$-structure of $E(2n)_m \# (20n-k)\overline{\mathbb{CP}^2}$ regarded as a blowup of $E(2n)_m$. One easily checks that the spin$^c$-structures $\{ \mathfrak{s}_m \}$ are all distinct and the same calculation as above gives $FBF_{X , \mathfrak{s}_m}(t_m) \neq 0$.
\end{proof}

\begin{remark}
In the proof of case (1) of Theorem \ref{thm:exotic1}, the spin$^c$-structures $\{ \mathfrak{s}_m \}$ that we constructed are not only distinct, but they lie in different $Diff(X)$-orbits because $c(\mathfrak{s}_m)$ is $(2n-1)m-1$ times a primitive class. This does not hold in case (2), except when $k = 20n$.
\end{remark}

Let $X$ be a compact, oriented, smooth, simply-connected $4$-manifold. Fix an embedding $D \to X$ of a closed disc in $X$ and let $Diff(X,D)$ denote the subgroup of $Diff(X)$ whose restriction to $D$ is the identity. The relative mapping class group $M_0(X)$ of $X$ is defined to be $M_0(X) = \pi_0(Diff(X,D))$. The natural map $Diff(X , D) \to Diff(X)$ induces a surjection $M_0(X) \to M(X)$. In the case that $X$ is spin $M_0(X)$ is an extension of $M(X)$ by $\mathbb{Z}_2$ and in the case that $X$ is not spin, $M_0(X) \to M(X)$ is an isomorphism \cite{gir}. We define the relative Torelli group $T_0(X)$ to be the pullback of $M_0(X)$ under $T(X) \to M(X)$. Equivalently, $T_0(X)$ is the subgroup of $M_0(X)$ which acts trivially on $H^2(X ; \mathbb{Z})$. Let $W$ be another compact, oriented, smooth, simply-connected $4$-manifold and fix an embedding $D' \to W$ of a closed disc in $W$. Form the connected sum $X \# W$ by removing open balls in $D$ and $D'$ and identifying their boundaries. Then given a diffeomorphism $f \in Diff(X,D)$, the connected sum $f \# id_W$ is well defined. Clearly the isotopy class of $f \# id_W$ depends only on the isotopy class of $f$ in $Diff(X,D)$, so we have a well-defined stablisation map
\[
\# id_{W} : M_0(X) \to M( X \# W),
\]
which restricts to a map
\[
\# id_{W} : T_0(X) \to T(X \# W).
\]

Let $W$ be a compact, oriented, smooth, simply-connected $4$-manifold. For the purpose of the next theorem, it will be conventient to call such a $4$-manifold {\em Bauer--Furuta non-degenerate} if $W$ admits a spin$^c$-structure $\mathfrak{s}_W$ with $d(\mathfrak{s}_W) = 0$ or $1$ and $BF_{W,\mathfrak{s}_W} = \eta$ or $\eta^2$.

\begin{theorem}\label{thm:exotic2}

Let $X$ be one of:
\begin{itemize}
\item[(1)]{$n (S^2 \times S^2) \# n K3$, $n \ge 1$, or}
\item[(2)]{$4 n \mathbb{CP}^2 \# k \overline{\mathbb{CP}^2}$, $n\ge 1$, $k \ge 20n$.}
\end{itemize}

Then there exists a homomorphism
\[
\Phi : T(X) \to (\Omega^{fr}_1)^\infty
\]
whose image is not finitely generated and has the following property. Let $W$ be a compact, oriented, smooth, simply-connected $4$-manifold which is Bauer--Furuta non-degenerate. Then there exists $d \in \{0,1\}$ and a homomorphism
\[
\Psi : T(X \# W ) \to (\Omega^{fr}_{d+2})^\infty
\]
such that the following diagram commutes
\[
\xymatrix{
T_0(X) \ar[rr]^-{ \# id_W } \ar[d] & & T(X \# W) \ar[d]^-{\Psi} \\
T(X) \ar[r]^-{\Phi} & (\Omega^{fr}_1)^\infty \ar[r]^-{\eta^{d+1} } & (\Omega^{fr}_{d+2})^\infty
}
\]

In particular, $X$ admits infinitely many isotopy classes of exotic diffeomorphisms which remain exotic on connected sums with any $4$-manifold $W$ which is Bauer--Furuta non-degenerate.
\end{theorem}
\begin{proof}
The result follows easily from Theorem \ref{thm:exotic1} and the gluing theorem for the families Bauer--Furuta invariant. Let $\Phi : T(X) \to \bigoplus_{\mathfrak{s}} \Omega^{fr}_1$ be as in Theorem \ref{thm:exotic1}. Since $W$ is assumed to be Bauer--Furuta non-degenerate, there is a spin$^c$-structure $\mathfrak{s}_W$ such that $BF_{W , \mathfrak{s}_W} = \eta^{1+d}$, where $d = d(\mathfrak{s}_W) \in \{0,1\}$. Define $\Psi : T(X \# W) \to \bigoplus_{\mathfrak{s}} \Omega^{fr}_{2+d}$ to be $\Psi = \bigoplus_{\mathfrak{s}} FBF_{X \# W , \mathfrak{s} \# \mathfrak{s}_W}$, where the sum is over spin$^c$-structures on $X$ with $d(\mathfrak{s}) = -1$. Proposition \ref{prop:gluing2} implies that $\Psi( f \# id_W)  = \eta^{d+1} \cdot \Phi(f)$.
\end{proof}

\begin{remark}
Let $X'$ be a compact, oriented, smooth, simply-connected $4$-manifold with $b_+(X) = 3 \; ({\rm mod} \; 4)$. Suppose $X'$ has a spin$^c$-structure $\mathfrak{s}'$ such that $d(\mathfrak{s}') = 0$ and $SW(X' , \mathfrak{s}')$ is odd. Then $BF_{X' , \mathfrak{s}'} = \eta$ \cite[Proposition 4.4]{bf} and hence $X'$ is Bauer--Furuta non-degenerate. In fact the condition $d(\mathfrak{s}') = 0$ is superfluous. The conditions $b_+(X') = 3 \; ({\rm mod} \; 4)$ and $SW(X' , \mathfrak{s}') = 1 \; ({\rm mod} \; 2)$ automatically implies that $d(\mathfrak{s}) = 0$ by \cite[Theorem 3.7]{bf}. If $X''$ is another $4$-manifold satisfying these conditions then $X' \# X''$ is also Bauer--Furuta non-degenerate because of the connected sum formula $BF_{X' \# X'' , \mathfrak{s}' \# \mathfrak{s}''} = BF_{X' , \mathfrak{s}' } \wedge BF_{X'' , \mathfrak{s}''}$.
\end{remark}

\begin{theorem}

Let $Z$ be one of:
\begin{itemize}
\item[(1)]{$n (S^2 \times S^2) \# (n+1) K3$, $n \ge 1$, or}
\item[(2)]{$(4 n+3) \mathbb{CP}^2 \# k \overline{\mathbb{CP}^2}$, $n\ge 1$, $k \ge 20n+19$.}
\end{itemize}

Then there exists a surjective homomorphism
\[
\varphi : T(Z) \to \mathbb{Z}_2^\infty.
\]
In particular, the Torelli group $T(Z)$ is not finitely generated.
\end{theorem}
\begin{proof}
Take $Z = X \# W$, where $X$ is as in Theorem \ref{thm:exotic2} and $W = K3$. The image of $\Psi : T(Z) \to (\Omega^{fr}_2)^\infty$ is not finitely generated, so $\Psi$ factors through a surjection $\varphi : T(Z) \to \mathbb{Z}_2^\infty$.
\end{proof}

\section{Non-finitely generated mapping class groups}\label{sec:mcg}

In \cite{bar2}, we proved that for certain simply-connected $4$-manifolds $X$, the mapping class group $M(X)$ is not finitely generated. Namely $X$ was taken to be $E(n) \# (S^2 \times S^2) = 2n \mathbb{CP}^2 \# 10n \overline{\mathbb{CP}}^2$ for any odd $n \ge 3$. The same result was shown by Konno for $E(n) \# (S^2 \times S^2)$ and $n \ge 2$ \cite{kon}. In \cite{bt}, we extended this result to the case $n=1$. In this section, we will prove that $M(X)$ is not finitely generated for a different class of simply-connected $4$-manifolds.

The strategy used in \cite{bar2,bt} was to contruct an infinite collection of homomorphisms $M_+(X) \to \mathbb{Z}$ which were contructed by summing the families Seiberg--Witten invariants over a $Diff(X)$-invariant set spin$^c$-structures. It is not clear to us whether the same strategy can be made to work for the families Bauer--Furuta invariant and so we will use a different approach, although the underlying idea is again to sum over spin$^c$-structures.

Let $X$ be a compact, oriented, simply-connected smooth $4$-manifold. Let $\mathfrak{s}$ be a spin$^c$-structure on $X$ with $d(\mathfrak{s}) = -1$ or $0$, so that $\Omega^{fr}_{d(\mathfrak{s})+2} \cong \mathbb{Z}_2$ is $2$-torsion. Recall from Proposition \ref{prop:diffinv} that the families Bauer--Furuta invariant $FBF_{X , \mathfrak{s}} : T(X) \to \Omega^{fr}_{d(\mathfrak{s})+2}$ is $M(X)$-invariant in the sense that
\[
FBF_{X , h \cdot \mathfrak{s} }( h \circ f \circ h^{-1} ) = FBF_{X , \mathfrak{s}}(f)
\]
for all $f \in T(X)$, $h \in M(X)$ (no factor of $sgn_+(h)$ is needed because $\Omega^{fr}_{d(\mathfrak{s})+2}$ is $2$-torsion). Let $M(X)$ act on $T(X)$ by conjugation, so that the above equality can be written as $FBF_{X , h \cdot \mathfrak{s}}( h \cdot f ) = FBF_{X , \mathfrak{s}}(f)$. Let $m \ge 1$ and let $\Delta_{m,d}$ be the set of isomorphism classes of spin$^c$-structures $\mathfrak{s}$ on $X$ such that $d(\mathfrak{s}) = d$ and $c(\mathfrak{s})$ is $m$ times a primitive class in $H^2(X ; \mathbb{Z})$. Assume $d \in \{-1,0\}$. Then we obtain a $M(X)$-invariant homomorphism $T(X) \to \Omega^{fr}_{d+2}$ by summing $FBF_{X,\mathfrak{s}}$ over all $\mathfrak{s}$ in $\Delta_{m,d}$. Unfortunately the resulting homomorphism is zero. This is because spin$^c$-structures in $\Delta_{m,d}$ come in charge conjugate pairs $\{ \mathfrak{s} , -\mathfrak{s} \}$ and $FBF_{X , \mathfrak{s}} = \pm FBF_{X  , -\mathfrak{s}}$ by the charge conjugation symmetry of the Seiberg--Witten equations. To get around this problem, introduce an equivalence relation $\sim$ on $\Delta_{m,d}$ whose equivalence classes are charge conjugate pairs. The families Bauer--Furuta invariant thought of as the map $FBF_X : \Delta_{m,d} \times T(X) \to \Omega^{fr}_{d+2}$ sending $( \mathfrak{s} , f)$ to $FBF_{X,\mathfrak{s}}(f)$ descends to a map $FBF_X : (\Delta_{m,d}/\! \! \sim ) \times T(X) \to \Omega^{fr}_{d+2}$. We then define $\Psi_{m,d} : T(X) \to \Omega^{fr}_{d+2}$ to be
\[
\Psi_{m,d}(f) = \sum_{ \mathfrak{s} \in \Delta_{m,d}/\! \sim} FBF_{X,\mathfrak{s}}(f).
\]

Clearly $\Psi_{m,d}$ is a group homomorphism and from the diffeomorphism invariance of $FBF$, we have that $\Psi_{m,d}$ is $M(X)$-invariant in the sense that
\begin{equation}\label{equ:invar}
\Psi_{m,d}( h \cdot f) = \Psi_{m,d}(f), \text{ for all } f \in T(X), h \in M(X).
\end{equation}

\begin{lemma}\label{lem:psi}
Let $X$ be one of 
\begin{itemize}
\item[(1)]{$4 n \mathbb{CP}^2 \# 20n \overline{\mathbb{CP}^2}$, $n\ge 1$.}
\item[(2)]{$(4 n-1) \mathbb{CP}^2 \# (20n-1) \overline{\mathbb{CP}^2}$, $n\ge 2$.}
\end{itemize}

In case (1) there exists infinitely many $m$ such that $\Psi_{m,-1} \neq 0$. In case (2) there exist infinitely many $m$ such that $\Psi_{m,0} \neq 0$.

\end{lemma}
\begin{proof}
Consider case (1) first. Write $X = X_0 \# (S^2 \times S^2)$, where $X_0 = (4n-1)\mathbb{CP}^2 \# (20n-1)\overline{\mathbb{CP}^2}$. Then $L = H^2(X  ;\mathbb{Z})$ splits as $L = L_0 \oplus H$, where $L_0 = H^2(X_0 ; \mathbb{Z})$ and $H = H^2(S^2 \times S^2 ; \mathbb{Z})$. We will show that $\Psi_{m , -1} \neq 0$ whenever $m = (2n-1)m'-1$ for some even $m'$. In fact, letting $f_{m'} : X \to X$ be the diffeomorphisms constructed in Theorem \ref{thm:exotic1} (in case (2) with $k = 20n$), we will show that $\Psi_{m,-1}(t_{m'}) \neq 0$, where $t_{m'} = f_{m'} \circ f_0$. Let $\Psi^{pin}_{m,-1}$ denote the composition of $\Psi_{m,-1}$ with the map $\Omega^{fr}_1 \to \Omega^{pin}_1$. It suffices to show that $\Psi^{pin}_{m,-1}(t_{m'}) \neq 0$.

Any spin$^c$-structure $\mathfrak{s} = \Delta_{m,-1}$ can be written as $\mathfrak{s} = \mathfrak{s}_0 \# \mathfrak{s}_1$, where $\mathfrak{s}_0$ is a spin$^c$-structure on $X_0$ and $\mathfrak{s}_1$ is a spin$^c$-structure on $S^2 \times S^2$. We will divide the spin$^c$-structures in $\Delta_{m,-1}$ into two cases: (i) $c(\mathfrak{s}_1) = 0$ and (ii) $c(\mathfrak{s}_1) \neq 0$. In case (i), both $f_0$ and $f_m$ preserve $\mathfrak{s}$ and thus
\begin{align*}
FBF^{pin}_{X , \mathfrak{s}}( t_{m'}) &= FBF^{pin}_{X,\mathfrak{s}}(f_{m'} \circ f_0) \\
&= FBF^{pin}_{X,\mathfrak{s}}(f_{m'}) + FBF^{pin}_{X,\mathfrak{s}}(f_0) \\
&= SW( E(2n)_{m'} , \mathfrak{s}_1) \eta^{pin}.
\end{align*}

Now if $c(\mathfrak{s})$ has divisibility $m = (2n-1)m' - 1$, then $c(\mathfrak{s}_1)$ also has divisibility $m$. Then $SW( E(2n)_{m'} , \mathfrak{s}_1) = 1 \; ({\rm mod} \; 2)$ for $\mathfrak{s}_1 = \pm \mathfrak{s}_{can}$ and $SW( E(2n)_{m'} , \mathfrak{s}_1) = 0$ otherwise. Thus, out of all the charge conjugate pairs of spin$^c$-structures in $\Delta_{m,-1}/\! \sim$ of type (i), only $\mathfrak{s}_{can} \# \mathfrak{s}_1$ (where $c(\mathfrak{s}_1) = 0$) contributes to $\Psi^{pin}_{m,-1}(t_{m'})$.

Now consider case (ii), this $\mathfrak{s} = \mathfrak{s}_0 \# \mathfrak{s}_1$, where $c(\mathfrak{s}_1) \neq 0$. Note also that $c(\mathfrak{s}_0) \neq 0$ because the lattice $L_0$ is odd. Such spin$^c$-structures can be partitioned into groups of four: $\{ \mathfrak{s} , -\mathfrak{s} , \mathfrak{s}' , -\mathfrak{s}' \}$, where $\mathfrak{s}' = \mathfrak{s}_0 \# (-\mathfrak{s}_1)$. This group of four is made up of two charge conjugate pairs $\{ \pm \mathfrak{s} \}, \{ \pm \mathfrak{s}' \}$. We will show that $FBF_{X , \mathfrak{s}}(t_{m'}) = FBF_{X , \mathfrak{s}'}(t_{m'})$. Thus the contributions to $\Psi^{pin}_{m,-1}(t_{m'})$ coming from $\mathfrak{s}$ and $\mathfrak{s}'$ cancel out and hence $\Psi^{pin}_{m,-1}(t_{m'}) = \eta^{pin}$.

Observe that $f_0 \cdot \mathfrak{s}' = \mathfrak{s}$ and thus
\[
FBF^{pin}_{X , \mathfrak{s}'}( t_{m'} ) = FBF^{pin}_{X , \mathfrak{s}}( f_0 \circ t_{m'} \circ f_0^{-1} ) = FBF^{pin}_{X , \mathfrak{s}}( f_0 \circ f_{m'}).
\]
Therefore,
\begin{align*}
FBF^{pin}_{X,\mathfrak{s}}(t_{m'}) + FBF^{pin}_{X,\mathfrak{s}'}(t_{m'}) &= FBF^{pin}_{X,\mathfrak{s}}( f_{m'} \circ f_0 ) + FBF^{pin}_{X , \mathfrak{s}}( f_0 \circ f_{m'}) \\
&= FBF^{pin}_{X , \mathfrak{s}}( f_{m'} \circ f_0^2 \circ f_{m'}). 
\end{align*}

We have that $f_0^2$ is isotopic to the identity. This follows because $f_0 = id_{X_0} \# \rho$, where $\rho$ is isotopic to an odd involution on $S^2 \times S^2$. Then $f_0^2$ is isotopic to a Dehn twist on the neck of $X_0 \# (S^2 \times S^2)$. But such a Dehn twist is isotopic to the indentity since we can use a circle action on $S^2 \times S^2$ to undo the twist. Similar reasoning shows that $f_{m'}^2$ is isotopic to the identity. Hence
\[
FBF^{pin}_{X , \mathfrak{s}}( f_{m'} \circ f_0^2 \circ f_{m'}) = FBF^{pin}_{X , \mathfrak{s}}( f_{m'}^2) = 0.
\]

This proves that the contrubtions to $\Psi^{pin}_{m,-1}(t_{m'})$ coming from $\mathfrak{s}$ and $\mathfrak{s}'$ cancel, as claimed.

The proof in case (2) is almost identical except that now $X = X_{n-1} \# K3$, where $X_{n-1} = 4(n-1)\mathbb{CP}^2 \# 20(n-1)\overline{\mathbb{CP}^2}$ is the manifold in case (1), but with $n$ replaced by $n-1$. Now instead of the diffeomorphisms $t_{m'} = f_{m'} \circ f_0$, we replace $t_{m'}$ by $t_{m'} \# id_{K3}$, $f_{m'}$ by $f_{m'} \# id_{K3}$ and $f_0$ by $f_0 \# id_{K3}$.
\end{proof}

\begin{theorem}

Let $X$ be one of 
\begin{itemize}
\item[(1)]{$4 n \mathbb{CP}^2 \# 20n \overline{\mathbb{CP}^2}$, $n\ge 1$.}
\item[(2)]{$(4 n-1) \mathbb{CP}^2 \# (20n-1) \overline{\mathbb{CP}^2}$, $n\ge 2$.}
\end{itemize}

Then $M(X)$ is not finitely generated. In fact, the abelianisation of $M(X)$ is not finitely generated.

\end{theorem}
\begin{proof}
For a group $G$, let $G_{ab}$ denote its abelianisation. The $M(X)$-action on $T(X)$ by conjugation descends to an action on $T(X)_{ab}$. Let $(T(X)_{ab})_{M(X)}$ denote the coinvariants of this action.

Since $\Omega^{fr}_{d+2}$ is an abelian group, the homomorphism $\Psi_{m,d} : T(X) \to \Omega^{fr}_{d+2}$ descends to $\Psi_{m,d} : T(X)_{ab} \to \Omega^{fr}_{d+2}$. Furthermore (\ref{equ:invar}) says that $\Psi_{m,d}$ is $M(X)$-invariant, so $\Psi_{m,d}$ descends to a homomorphism $\Psi_{m,d} : (T(X)_{ab})_{M(X)} \to \Omega^{fr}_{d+2} \cong \mathbb{Z}_2$. By Proposition \ref{prop:compact1}, for fixed $f \in T(X)$, $\Psi_{m,d}(f) = 0$ for all but finitely many $m$. So for fixed $d \in \{-1,0\}$, we obtain a homomorphism
\[
\Psi : (T(X)_{ab})_{M(X)} \to \bigoplus_{m \ge 1} \mathbb{Z}_2, \quad \Psi = \bigoplus_m \Psi_{m,d}.
\]
By Lemma \ref{lem:psi}, the image of $\Psi$ is not finitely generated and hence $(T(X)_{ab})_{M(X)}$ is not finitely generated.

Recall the short exact sequence $1 \to T(X) \to M(X) \to \Gamma \to 1$. From the Lyndon--Hochschild--Serre spectral sequence in homology, we get the exact sequence
\[
H_2( \Gamma ; \mathbb{Z}) \buildrel \partial \over \longrightarrow (T(X)_{ab})_{M(X)} \to M(X)_{ab} \to \Gamma_{ab} \to 0.
\]

Now since $X$ is of the form $X = X_0 \# (S^2 \times S^2)$, where $X_0$ is simply-connected and $H^2(X_0 ; \mathbb{Z})$ is indefinite, it follows from \cite[Theorem 2]{wall} that $\Gamma$ is the group of automorphisms of the lattice $L = H^2(X ; \mathbb{Z})$. Then $\Gamma$ is an arithmetic group and is finitely generated \cite[Theorem 6]{sou}. It follows that $H_2(\Gamma ; \mathbb{Z})$ is finitely-generated \cite[Exercise 5(a), Page 46]{bro}. If $(T(X)_{ab})_{M(X)}/\partial( H_2(\Gamma ; \mathbb{Z}))$ was finitely generated, then $(T(X)_{ab})_{M(X)}$ would also be finitely generated, which we have shown above is not the case. Therefore $(T(X)_{ab})_{M(X)}/\partial( H_2(\Gamma ; \mathbb{Z}))$ is a non-finitely generated subgroup of the abelian group $M(X)_{ab}$. Hence $M(X)_{ab}$ is not finitely generated and so neither is $M(X)$.
\end{proof}


\bibliographystyle{amsplain}

\end{document}